\documentclass[12pt]{amsart}
\usepackage[headings]{fullpage}
\usepackage{amssymb,epic,eepic,epsfig,amsbsy,amsmath,amscd,color,amsthm}

\usepackage{amsmath,amsthm}
\usepackage{verbatim}
\usepackage{marvosym}
\usepackage{graphicx}
\usepackage{color}
\usepackage{epsfig}
\usepackage{rotating}
\usepackage{lscape}
\usepackage{hyperref}
\usepackage[all]{xy}
\usepackage{mathrsfs, euscript, mathdesign}

\usepackage{psfrag}
\numberwithin{equation}{section}
                        \theoremstyle{plain}
\usepackage{mathrsfs}

\newtheorem{theorem}{Theorem}[section]

\newtheorem{lemma}[theorem]{Lemma}

\newtheorem{proposition}[theorem]{Proposition}

\newtheorem{definition}[theorem]{Definition}

\theoremstyle{definition}
\newtheorem{remark}[theorem]{Remark}

\makeatletter
\newtheorem*{rep@theorem}{\rep@title}
\newcommand{\newreptheorem}[2]{%
\newenvironment{rep#1}[1]{%
 \def\rep@title{#2 \ref{##1}}%
 \begin{rep@theorem}}%
 {\end{rep@theorem}}}
\makeatother

\newreptheorem{theorem}{Theorem}

\def\BP{\mathbb P}
\def\BC{\mathbb C}
\def\BF{\mathbb F}

\def\BZ{\mathbb Z}

\def\BQ{\mathbb Q}

\def\CZ{\mathcal Z}

\def\PSL{\mathrm{PSL}}
\def\SL{\mathrm{SL}}

\def\fp{\mathfrak p}

\def\la{\langle}
\def\ra{\rangle}

\def\om{\omega}
\def\si{\sigma}

\DeclareMathOperator{\tr}{\mathrm tr}

\def\be { \begin{equation} }
\def\ee { \end{equation} }

\newcounter{nootje}
\begin{document}

\title[Character varieties of double twist links]
{Character varieties of double twist links}

\author[Kathleen L. Petersen]{Kathleen L. Petersen}
\address{Department of Mathematics, Florida State University, Tallahassee, TX 32306, USA}
\email{petersen@math.fsu.edu}

\author[Anh T. Tran]{Anh T. Tran}
\address{Department of Mathematical Sciences, The University of Texas at Dallas, Richardson TX 75080, USA}
\email{att140830@utdallas.edu}

\begin{abstract}
We compute both natural and smooth models for the $\SL_2(\BC)$ character varieties of the two component double twist links, an infinite family of two-bridge links indexed as $J(k,l)$.  For each $J(k,l)$, the component(s) of the character variety containing characters of irreducible representations are birational to a surface of the form $C\times \BC$ where $C$ is a curve. The same is true of the canonical component.
We compute the genus of this curve, and the degree of irrationality of the canonical component. We realize the natural model of the canonical component of the $\SL_2(\BC)$ character variety of the $J(3,2m+1)$ links as the surface obtained from $\BP^1\times \BP^1$ as a series of blow-ups. 
\end{abstract}

\thanks{2010 {\em Mathematics Classification:} Primary 57M25. Secondary 57N10, 14J26.\\
{\em Key words and phrases: character variety, canonical component, double twist link.}}

\maketitle







\section{Introduction}

Given a complete orientable finite volume hyperbolic 3-manifold
with cusps, the $\SL_2(\BC)$ character variety of $M$, $X(M)$, is a
complex algebraic set associated to representations of $\pi_1(M) \rightarrow \SL_2(\BC)$.  Thurston \cite{Thurston} showed 
that
any irreducible component of such a variety containing the character
of a discrete faithful representation has complex dimension equal to
the number of cusps of $M$.  Such components are called {\em canonical
  components} and are denoted $X_0(M)$.  
Character varieties have been fundamental
tools in studying the topology of $M$ (we refer the reader to
\cite{MR1886685} for more), and 
canonical components encode a
wealth of topological information about $M$, including containing 
subvarieties associated to Dehn fillings of $M$ and identifying boundary slopes of essential surfaces \cite{MR683804}.

\begin{figure}[h!]
\begin{center}
\includegraphics[scale=.5]{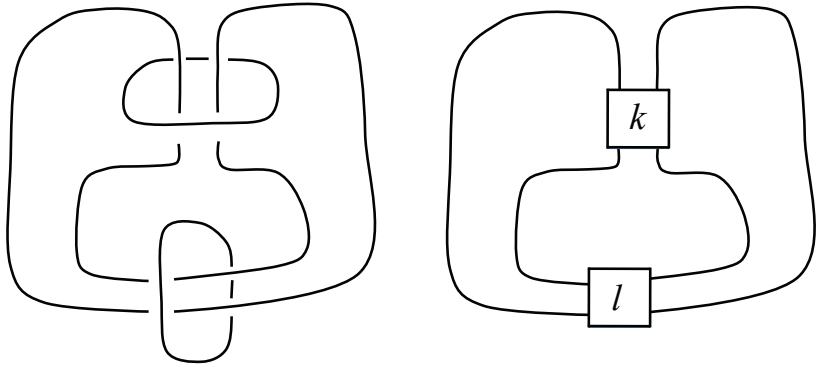}
\end{center}
\vspace{-5cm}
\caption{The link $J(k,l)$ is the result of $-1/k$ and $-1/l$ surgery on the three component link pictured on the left.}
\label{fig:threecomponents}
\end{figure}

In this paper, we consider the two component double twist links $J(k,l)$ and compute the character varieties of their complements in $S^3$.    As pictured in Figure~\ref{fig:threecomponents} the integers $k$ and $l$ determine the number of half twists in the boxes; positive
numbers correspond to right-handed twists and negative numbers
correspond to left-handed twists. The link $J(k,l)$ is a  two component link when $kl$ is odd and a knot  when $kl$ is even.   In \cite{MR2827003} character varieties of the $J(k,l)$ knots were determined and analyzed.  In this paper, we extend this work to the two component $J(k,l)$ links. These are hyperbolic exactly when $|k|$ and $|l|$ are greater than one; the $J(\pm 1,l)=J(l,\pm 1)$ links are torus links.  We will now exclusively consider the hyperbolic $J(k,l)$ links.

In Definition~\ref{definition:chebyshev} we define the Chebyshev polynomials $S_j$ which are used throughout the paper.  Our first theorem establishes natural models for the $\SL_2(\BC)$  character varieties of the double twist links. With $\pi_1(k,l)=\pi_1(S^3-J(k,l))$,   let $X_{irr}(k,l)$ denote the closure of the set of all irreducible characters $\chi_{\rho}$ such that $\rho:\pi_1(k,l)\rightarrow \SL_2(\BC)$. Let $X_0(k,l)$ denote a canonical component. 

\begin{theorem}\label{theorem:naturalmodel} Let $k=2m+1$ and $l=2n+1$.  Using the presentation for $\pi_1(k,l)$ in Section~\ref{section:doubletwistlinks} with $x=\chi_{\rho}(a)$, $y=\chi_{\rho}(b)$ and $z=\chi_{\rho}(ab^{-1})$, we have the following.

The vanishing set of $xyz+4-x^2-y^2-z^2$ in $\BC^3(x,y,z)$ is the set of characters of reducible representations of $\pi_1(k,l)\rightarrow \SL_2(\BC)$.  

A natural model for the algebraic set $X_{irr}(k,l)$ is the vanishing set of 
\[ S_n(t)S_{m-1}(z)-S_{n-1}(t)S_m(z)\] in $\BC^3(x,y,z)$ where 
\[t=\left( xS_m(z)-yS_{m-1}(z) \right)\left( yS_m(z)-xS_{m-1}(z) \right)-z\left(S^2_m(z)+S^2_{m-1}(z)\right)+4S_m(z)S_{m-1}(z).\]
\end{theorem}

Our next theorem establishes smooth models for these algebraic sets. 
\begin{theorem}\label{theorem:summary1}  Let $k=2m+1$ and $l=2n+1$. 
The algebraic set $X_{irr}(k,l)$ is  birational to $C(k,l)\times \BC$ where the curve $C(k,l)\subset \BC^2(t,z)$ is given by 
\[ C(k,l) = \{S_n(t)S_{m-1}(z)-S_{n-1}(t)S_m(z)=0\}.\]

If $k\neq l$ then $C(k,l)$ is smooth and irreducible as considered in $\BP^1(t)\times \BP^1(z)$, and $X_0(k,l)=X_{irr}(k,l)$ is birational to $C(k,l)\times \BC$.  

The curve $C(3,3)=C(-3,-3)$ is given by $t=z$.  
If $k=l$ and $|l|>3$ then $C(l,l)$ is the union of exactly two components, $C_0(l,l)$ given by $t=z$ and $C_1(l,l)$, the scheme-theoretic complement of $C_0(l,l)$ in $C(l,l)$. Both are smooth and irreducible  as considered in $\BP^1(t)\times \BP^1(z)$. The algebraic set $X_{irr}(k,l)= X_0(l,l) \cup X_1(l,l)$ where  $X_0(l,l)$ is birational to  $ C_0(l,l)\times \BC$ and $X_1(l,l)$ is birational to $ C_1(l,l)\times \BC$.  \end{theorem}

We next compute some invariants of these algebraic sets. Since $X_{irr}(k,l)$ is birational to the product of a curve $C(k,l)$ and $\BC$ we compute the genus of this curve. 

\begin{theorem}\label{theorem:genus} 
Let $k=2m+1$ and $l=2n+1$ with $|k|, |l|>1$.
When $k\neq l$  the genus of $C(k,l)$ is \[(\lfloor  \tfrac{|k|}2 \rfloor -1)( \lfloor  \tfrac{|l|}2 \rfloor-1).\]  The genus of $C_0(l,l)$ is zero, and when $|l|>3$ the genus of $C_1(l,l)$ is  $(\lfloor  \tfrac{|l|}2 \rfloor -2)^2$.
\end{theorem}

The degree of irrationality of an irreducible $n$-dimensional complex algebraic set $X$ is defined to be 
the minimal degree of any rational map from $X$ to a dense subset of $ \BC^n$. This is  denoted $\gamma(X)$ and is a birational invariant. When $X$ is a curve this is called the gonality of $X$.  See \cite{PetersenReid} for a discussion on how gonality and genus behave in families of Dehn fillings. In light of this, since $J(k,l)$ is $-1/k$ and $-1/l$ filling of the three component link in Figure~\ref{fig:threecomponents} we compute the degree of irrationality of the surfaces $X_0(k,l)$ and $X_1(l,l)$.

\begin{theorem}\label{theorem:degreeofirrationality}
Let $k=2m+1$ and $l=2n+1$. 
The degree of irrationality of $X_0(k,l)$ is $\min\{  \lfloor  \tfrac{|k|}2 \rfloor, \lfloor  \tfrac{|l|}2 \rfloor \}$ when $k\neq l$.  The degree of irrationality of $X_0(l,l)$ is 1, and when $|l|>3$ the degree of irrationality of $X_1(l,l)$ is $\lfloor  \tfrac{|l|}2 \rfloor -1$.
\end{theorem}

Finally, we study that the $J(3,2m+1)$ links realizing $X_0(3,2m+1)$ as a series of blow ups of $\BP^1\times \BP^1$ and show the following.

\begin{theorem}\label{theorem:blowup}
The desingularization of the natural model for the canonical component of the $\SL_2(\BC)$-character variety of the double twist link $J(3,2m+1)$ is the conic bundle over the projective line $\BP^1$ which is isomorphic to the surface obtained from $\BP^1 \times \BP^1$ by repeating a one-point blow up $9m$ times if $m \ge 1$, and $-(6+9m)$ times if $m \le -2$. Equivalently,  it is isomorphic to the surface obtained from  $\BP^2$ by repeating a one-point blow up $1+9m$ times if $m \ge 1$, and $-(5+9m)$ times if $m \le -2$.
\end{theorem}

\begin{remark}
For $m \ge 1$, the link $J(3,2m+1)$ is obtained by $1/m$ Dehn surgery on the Magic manifold. Hence Theorem \ref{theorem:blowup} confirms Conjecture 3.1.3 in Landes' thesis \cite{Landes}. 
\end{remark}

\section*{Acknowledgement}
This work was partially supported by a grant from the Simons Foundation (\#209226 to Kathleen Petersen).

\section{Character Varieties}

We will define our notations, but refer the reader to \cite{MR2827003} for a detailed discussion of character varieties.  Let $M$ be a complete finite volume  hyperbolic 3-manifold.    The $\SL_2(\BC)$ character variety of $M$ is the set of all characters of representations $\rho:\pi_1(M) \rightarrow \SL_2(\BC)$.  The character associated to $\rho$ is $\chi_{\rho}: \pi_1(M) \rightarrow \BC$ defined by $\chi_{\rho}(\gamma) = \tr \rho(\gamma)$. 

Let $X(M)$ denote the $\SL_2(\BC)$ character variety, that is 
\[ X(M)= \{ \chi_{\rho} \mid \rho:\pi_1(M) \rightarrow \SL_2(\BC)\}.\] The characters of reducible representations themselves form an algebraic set, which is a subset of $X(M)$.  We will call this set $X_{red}(M)$.  The closure of the set of characters of irreducible representations will be denoted by $X_{irr}(M)$.   Any irreducible component of $X(M)$ which contains the character of a discrete faithful representation is contained in $X_{irr}(M)$ and is called a canonical component and denoted $X_0(M)$. 

Thurston \cite{Thurston} showed that the complex dimension of any canonical component equals the number of cusps of $M$. The canonical components encode much of the topology of $M$, including containing subvarieties corresponding to Dehn fillings of $M$ (see \cite{PetersenReid}) and their ideal points can be used to determine essential surfaces in $M$ (see \cite{MR683804}).  

When $M$ has only one cusp $X_0(M)$ is a curve.  Several infinite families of these have been studied (see \cite{MR3078072}, \cite{MR2827003}, \cite{MR3103879}).  When $M$ has at least two cusps the algebraic geometry becomes more demanding, and only a few solitary examples have been computed.  Landes \cite{MR2831837, Landes} computed a smooth model for the canonical component of the $\SL_2(\BC)$ character variety of the complement of the Whitehead link, a two component link.   (She explicitly showed that it is a rational surface homeomorphic to the projective plane blown up at 10 points.) Harada \cite{harada} computed the character varieties of the four arithmetic two-bridge link complements (including the Whitehead link and the figure-8 knot).  Our computation of the character varieties of the double twist links is the first result to compute character varieties for infinitely many 3-manifolds with two cusps.

\section{Double Twist Links}\label{section:doubletwistlinks}

Let $J(k,l)$ be the double twist link indicated in the right hand side of  Figure~\ref{fig:threecomponents}.  This link is $-1/k$ and $-1/l$ filling on two components of the three component link shown in the left hand side of  Figure~\ref{fig:threecomponents}.  This is a knot when $kl$ is even and a two component link when $kl$ is odd.   
The link $J(k,l)$ corresponds to the continued fraction $[k,-l]$. It is hyperbolic, unless $|k|$ or $|l|$ is $1$.
Let  $X(k,l)$ denote   the $\SL_2(\BC)$  character variety of $S^3-J(k,l)$. 

In \cite{MR2827003} the character varieties of the $J(k,l)$ knots were computed.  We now consider the $J(k,l)$ links with two components, so  both $k$ and $l$ are odd. Suppose $k=2m+1$ and $l=2n+1$. The link group of $J(k,l)$ is  $\pi_1(k,l)=\pi_1(S^3-J(k,l))$ and has presentation
$$\pi_1(k,l)=\la a, b \mid aw_k^nb=w_k^{n+1}\ra$$
where $w_k=(ab^{-1})^mab(a^{-1}b)^m$ \cite{MR2827003}.

\begin{definition}
Let $F_{a,b}=\la a,b \ra$ be the free group in two letters $a$ and $b$. For a word $u$ in $F_{a,b}$ let $\overleftarrow{u}$ denote the word obtained from $u$ by writing the letters in $u$ in reversed order.
\end{definition}

We begin by simplifying the presentation of the link group.
\begin{lemma}\label{lemma:fundamentalgroup}
With  $w_k=(ab^{-1})^mab(a^{-1}b)^m$ and $r=w_k^n(ab^{-1})^m$, we have \[ \pi_1(k,l)=\la a,b \mid r= \overleftarrow{r} \ra.\]
\end{lemma}

\begin{proof}

We can rewrite the presentation of $\pi_1(k,l)$ as follows
\begin{eqnarray*}
\pi_1(k,l) &=& \la a, b \mid aw_k^nb=(ab^{-1})^mab(a^{-1}b)^mw_k^{n-1}(ab^{-1})^mab(a^{-1}b)^m\ra \\
      &=& \la a, b \mid w_k^n=(b^{-1}a)^mb(a^{-1}b)^mw_k^{n-1}(ab^{-1})^ma(ba^{-1})^m\ra \\
      &=& \la a, b \mid w_k^n(ab^{-1})^m=(b^{-1}a)^mb(a^{-1}b)^mw_k^{n-1}(ab^{-1})^ma\ra.
\end{eqnarray*}

Letting $c=(ab^{-1})^ma$ and $d=b(a^{-1}b)^m$ then $w_k=cd$. It follows that $$b(a^{-1}b)^mw_k^{n-1}(ab^{-1})^ma=d(cd)^{n-1}c=(dc)^n=\overleftarrow{(cd)^n}=\overleftarrow{w_k^n}.$$
Hence 
\begin{eqnarray*}
\pi_1(k,l) &=& \la a,b \mid w_k^n(ab^{-1})^m = \overleftarrow{(ab^{-1})^m} \, \overleftarrow{w_k^n}\ra \\
      &=& \la a,b \mid w_k^n(ab^{-1})^m = \overleftarrow{w_k^n(ab^{-1})^m} \ra.
\end{eqnarray*}
Since $r=w_k^n(ab^{-1})^m$, the lemma follows.
\end{proof}

The character variety of $F_{a,b}$ is isomorphic to $\BC^3$ by the Fricke-Klein-Vogt theorem \cite{FrickeKlein, MR1508833}. 
Let $x=\tr \rho(a)$, $y=\tr \rho(b)$ and $z=\tr \rho(ab^{-1})$. Consider a word $u$ in $F_{a,b}$.
Define the polynomial $P_u \in \BC[x,y,z]$ to be $P_u(x,y,z)= \tr \rho(u)$.
It follows that for every word $u$ in $F_{a,b}$ the  polynomial $P_u$ is the unique polynomial such that for any representation $\rho: F_{a,b} \to \SL_2(\BC)$ we have $\tr \rho(u)=P_u (x,y,z)$.  

We now consider representations  $\rho:\pi_1(k,l) \rightarrow \SL_2(\BC)$.  By Lemma~\ref{lemma:fundamentalgroup} the group  $\pi_1(k,l)$ has a presentation with two generators and one relation and therefore is a quotient of $F_{a,b}$.  First, we establish some notations which we will use throughout the manuscript.
\begin{definition}
Let  $k=2m+1$ and $l=2n+1$. For $\rho:\pi_1(k,l) \rightarrow \SL_2(\BC)$ define 
\[ x=\tr \rho(a),~y=\tr\rho(b) \ \text{and}  \  z=\tr\rho(ab^{-1})\] and for a word $u$ in $F_{a,b}$ define the polynomial $P_u(x,y,z)= \tr\rho(u) \in  \BC[x,y,z]$.  Further, let   $t=P_{w_k}$ and 
 \[\varphi(x,y,z) = P_{rab}-P_{\overleftarrow{r}ab}.\]
\end{definition}
For every representation $\rho: \pi_1(k,l) \to \SL_2(\BC)$, we consider $x,y$ and $z$ as functions of $\rho$. Using the above two generator one relation presentation for $\pi_1(k,l)$ we conclude that $P_{rab}=P_{\overleftarrow{r}ab}$, which is simply $\varphi(x,y,z)=0$, in $X(k,l)$.   In fact, by \cite[Thm.1]{MR3103879} $X(k,l)$ is exactly the zero set of $\varphi(x,y,z)$.  (See also \cite[Thm.2.1]{MR2888942}.) Moreover,  because of the format of the defining word, $P_{\overleftarrow{r}ab}=P_{bar}$ \cite[Thm.1]{MR3103879}.  (That is, these polynomials in $\BC^3[x,y,z]$ are identical.)  Therefore, $\varphi(x,y,z) = P_{rab}-P_{bar}$. We summarize this discussion in the following proposition. 
\begin{proposition}
The polynomial $\varphi(x,y,z) = P_{rab}-P_{bar}$. 
The character variety $X(k,l)$ is the zero set of $\varphi(x,y,z)$ in $\BC^3(x,y,z)$. 
\end{proposition}

We wish to obtain a nice format for $\varphi$.  We introduce a family of Chebyshev polynomials, often called the Fibonacci polynomials, that will be essential to our computation of $\varphi$.  (These are slightly different polynomials than were used in \cite{MR2827003}; the indices are shifted by one.)
\begin{definition}\label{definition:chebyshev}
Let $S_j(\omega)$ be the Chebyshev polynomials defined by $S_0(\omega)=1,~S_1(\omega)=\omega$ and $S_{j+1}(\omega)=\omega S_j(\omega)-S_{j-1}(\omega)$ for all integers $j$.
\end{definition}
It is elementary to verify the following lemmas.
\begin{lemma}\label{lemma:chebyshevidentities}
With $\omega=\sigma+\sigma^{-1}$ we have 
\[ S_j(\omega) = \frac{\sigma^{j+1}-\sigma^{-j-1}}{\sigma-\sigma^{-1}}.\] 
The degree of $S_j$ is $j$ if $j> -1$ and $-j-2$ if $j<-1$.  
\end{lemma}

\begin{lemma}
Suppose the sequence $\{f_j\}_{j \in \BZ}$ satisfies the recurrence relation $f_{j+1}=\omega f_j-f_{j-1}$ for all integers $j$. Then $f_j=S_{j}(\omega)f_0-S_{j-1}(\omega)f_{-1}.$
\label{chev}
\end{lemma}

The following lemma can be verified by using Lemma \ref{lemma:chebyshevidentities}. 

\begin{lemma}\label{identities}
We have

a) $S^2_j(\omega)+S^2_{j-1}(\omega)-\omega S_j(\omega)S_{j-1}(\omega)=1$,

b) $S^2_{j}(\omega)-S^2_{j-1}(\omega)=S_{2j}(\omega)$ and

c) $S_{m-1}(\omega) \left( \omega+(\omega^2-4)S_{m-1}(\omega)S_{m}(\omega) \right)+S_m(\omega)=S_{3m}(\omega)$. 
\end{lemma}

We now simplify the polynomial $\varphi$ by writing the trace polynomials in terms of these Chebyshev polynomials.
\begin{proposition}\label{proposition:tchebyshev}
We have $$t=\left( xS_m(z)-yS_{m-1}(z) \right)\left( yS_m(z)-xS_{m-1}(z) \right)-z\left(S^2_m(z)+S^2_{m-1}(z)\right)+4S_m(z)S_{m-1}(z).$$
\end{proposition}

\begin{proof}
By definition, $t=P_{w_k}$. By applying Lemma \ref{chev} twice, we have  
\begin{eqnarray*}
P_{w_k} &=& P_{(ab^{-1})^mab(a^{-1}b)^m} \\
        &=& S^2_m(z)P_{ab}+S^2_{m-1}(z)P_{(ab^{-1})^{-1}ab(a^{-1}b)^{-1}}-S_m(z)S_{m-1}(z)(P_{(ab^{-1})^{-1}ab}+P_{ab(a^{-1}b)^{-1}})   \\
        &=& S^2_m(z)P_{ab}+S^2_{m-1}(z)P_{ba}-S_m(z)S_{m-1}(z)(P_{b^2}+P_{a^2}) \\        
        &=& (S^2_m(z)+S^2_{m-1}(z))(xy-z)-S_m(z)S_{m-1}(z)(x^2+y^2-4).
\end{eqnarray*}
The proposition follows.
\end{proof}

\begin{proposition}
The polynomial $\varphi(x,y,z) \in \BC^3[x,y,z]$ is 
$$\varphi(x,y,z)=\Big(xyz+4-x^2-y^2-z^2\Big)\Big(S_n(t)S_{m-1}(z)-S_{n-1}(t)S_m(z)\Big)$$
where $t$ is as in Proposition~\ref{proposition:tchebyshev}.
\end{proposition}

\begin{proof}
As mentioned above, by \cite[Thm.1]{MR3103879} $X(k,l)$ is the zero set of $\varphi(x,y,z)$ and $P_{\overleftarrow{r}ab}=P_{bar}$.
By applying Lemma \ref{chev} we have
\begin{eqnarray*}
P_{rab}-P_{bar} &=& P_{w_k^n(ab^{-1})^mab}-P_{baw_k^n(ab^{-1})^m} \\
                &=& S_n(t)(P_{(ab^{-1})^mab}-P_{ba(ab^{-1})^m}) - S_{n-1}(t)(P_{w_k^{-1}(ab^{-1})^mab}-P_{baw_k^{-1}(ab^{-1})^m}) \\
                &=& S_n(t)(P_{(ab^{-1})^mab}-P_{ba(ab^{-1})^m}) - S_{n-1}(t)(P_{(a^{-1}b)^m}-P_{ab(a^{-1}b)^m(ba)^{-1}})
\end{eqnarray*}
where
\begin{eqnarray*}
P_{(ab^{-1})^mab}-P_{ba(ab^{-1})^m} &=& S_m(z)(P_{ab}-P_{ba}) - S_{m-1}(z)(P_{(ab^{-1})^{-1}ab}-P_{ba(ab^{-1})^{-1}}) \\
&=& - S_{m-1}(z)(P_{b^2}-P_{baba^{-1}}) \\
&=& S_{m-1}(z)(xyz+4-x^2-y^2-z^2), \\
P_{(a^{-1}b)^m}-P_{ab(a^{-1}b)^m(ba)^{-1}} &=& S_m(z)(P_{1}-P_{ab(ba)^{-1}}) - S_{m-1}(z)(P_{(a^{-1}b)^{-1}}-P_{ab(a^{-1}b)^{-1}(ba)^{-1}}) \\
&=& S_{m}(z)(xyz+4-x^2-y^2-z^2).
\end{eqnarray*}
Hence 
$$P_{rab}-P_{bar}=(xyz+4-x^2-y^2-z^2)(S_n(t)S_{m-1}(z)-S_{n-1}(t)S_m(z)).$$
\end{proof}

The character variety $X(k,l)$ is clearly reducible.  The set of reducible characters, $X_{red}(k,l)$, can easily be determined, as in  \cite{MR3078072}, for example. We have the following, from which Theorem~\ref{theorem:naturalmodel} follows immediately. 
\begin{proposition}\label{proposition:naturalmodel} 
The vanishing set of $xyz+4-x^2-y^2-z^2$ in $\BC^3(x,y,z)$ is the set of characters of reducible representations of $\pi_1(k,l) \rightarrow\SL_2(\BC)$.  

A natural model for the algebraic set $X_{irr}(k,l)$ is the vanishing set of $S_n(t)S_{m-1}(z)-S_{n-1}(t)S_m(z)$ in $\BC^3(x,y,z)$ where $t$ is as in Proposition~\ref{proposition:tchebyshev}.
\end{proposition}

In light of this, we wish to understand the vanishing set of $S_n(t)S_{m-1}(z)-S_{n-1}(t)S_m(z)$.  The equation $S_n(t)S_{m-1}(z)=S_{n-1}(t)S_m(z)$ can be written as 
\[ \frac{S_n(t)}{S_{n-1}(t)} = \frac{S_m(z)}{S_{m-1}(z)} \]
when $S_{n-1}(t)S_{m-1}(z)\neq 0$, so we can think of it as lying in a product of projective lines.  We will make use of this approach when proving smoothness and irreducibility. 
\begin{definition}
Let  $V(k,l)$ be the vanishing set of $S_n(t)S_{m-1}(z)-S_{n-1}(t)S_m(z)$ in $\BC^3(x,y,z)$.

\end{definition}

By Propostion~\ref{proposition:naturalmodel} the components of $X(k,l)$ containing characters of irreducible representations, those included in $X_{irr}(k,l)$, are contained in  $V(k,l)$ and $V(k,l)$ is a natural model for this set.

\section{The structure of $V(k,l)$}\label{section:birational} 

The set $V(k,l)$ is the closure of the set of characters of irreducible representations.  The equation $S_n(t)S_{m-1}(z)-S_{n-1}(t)S_m(z)$ is relatively simple, except that $t$ itself is a function of the natural variables $x,y,$ and $z$.  Explicitly, 
 by Proposition~\ref{proposition:tchebyshev} 
$$t=\left( xS_m(z)-yS_{m-1}(z) \right)\left( yS_m(z)-xS_{m-1}(z) \right)-z\left(S^2_m(z)+S^2_{m-1}(z)\right)+4S_m(z)S_{m-1}(z).$$

We will show that there is a relatively simple model for $X_{irr}(k,l)$ up to birational equivalence.

\begin{definition}
Let $u=xS_m(z)-yS_{m-1}(z)$ and $v=yS_m(z)-xS_{m-1}(z)$. 
\end{definition}
It follows that  $$t=uv-z\big(S^2_m(z)+S^2_{m-1}(z)\big)+4S_m(z)S_{m-1}(z).$$
By the definitions of $u$ and $v$, $$x=\frac{uS_m(z)+vS_{m-1}(z)}{S^2_m(z)-S^2_{m-1}(z)} \qquad \text{and} \qquad y=\frac{vS_m(z)+uS_{m-1}(z)}{S^2_m(z)-S^2_{m-1}(z)}.$$
We will show that this substitution of $u$ and $v$ for $x$ and $y$ corresponds to a birational map, simplifying the definition of $t$.  Then we will show that substituting $t$ for $u$ is another birational map, thus eliminating the problem of having nested variables.  This has the fortunate consequence that the equation $S_n(t)S_{m-1}(z)-S_{n-1}(t)S_m(z)$ contains no $u$, so we can conclude that the algebraic set $V(k,l)$ is birational to the product of a curve and $\BC$.

\begin{definition}
Let $U(k,l)$ be the vanishing set of 
\[ S_n(t)S_{m-1}(z)-S_{n-1}(t)S_m(z) \]
in $\BC^3(u,v,z)$ where 
$$t=uv-z \big(S^2_m(z)+S^2_{m-1}(z)\big)+4S_m(z)S_{m-1}(z).$$
\end{definition}

Before showing that $V(k,l)$ is birational to $U(k,l)$ we prove a lemma.

\begin{lemma}\label{lemma:zeroset}
On $V(k,l)$, $S^2_m(z)-S^2_{m-1}(z)=0$ only for a set of codimension one. 
\end{lemma}

\begin{proof}
By definition, $S_j(z)$ is a Chebyshev polynomial, and by Lemma~\ref{identities} we have that $S^2_m(z)-S^2_{m-1}(z)=S_{2m}(z)$. Moreover,  letting $z=\sigma+\sigma^{-1}$ we can write 
\[ S_{2m}(\sigma+\sigma^{-1})= \frac{\sigma^{2m+1}-\sigma^{-2m-1}}{\sigma-\sigma^{-1}}.\]
Therefore, if $S_m^2(z)-S_{m-1}^2(z)=0$ then $\sigma^{2m+1}-\sigma^{-2m-1}=0$ and so $\sigma^{4m+2}=1$.
 It follows that $\sigma=e^{2 \pi i s/(4m+2)}=e^{\pi i s/(2m+1)}$ for some $0\leq s \leq 4m+2$.  When $s=2r$ is even,  ($1 \le r \le m$)
\[ z= \sigma+\sigma^{-1} = 2 Re(\sigma)= 2\cos \Big( \tfrac{2\pi r}{2m+1} \Big)\]
and $z$ is a root of $S_m(z)+S_{m-1}(z)$.
When $s=2r+1$ is odd ($0 \le r \le m-1$)
\[ z=2\cos \Big(\tfrac{(2r+1)\pi }{2m+1} \Big) \]
and $z$ is a root of $S_m(z)-S_{m-1}(z)$.

First, we will show that on $V(k,l)$, $S_m(z)-S_{m-1}(z)=0$ only for a set of dimension one. 
Note that $z=2\cos \big( \tfrac{(2r+1)\pi }{2m+1}\big)$, where $0 \le r \le m-1$. 
By Lemma~\ref{identities},  $S^2_m(z)+S^2_{m-1}(z)-zS_m(z)S_{m-1}(z)=1$. Since $S_m(z)=S_{m-1}(z)$, we obtain $S^2_m(z)=\frac{1}{2-z}$ and 
\begin{eqnarray*}
t &=&  -S_m^2(z) \Big(( x -y)^2  +2z -4  \Big) \\
  &=&   \frac{1}{z-2} \Big(( x -y)^2  +2(z -2)  \Big) \\
    &=&   \frac{( x -y)^2}{z-2}  +2.
\end{eqnarray*}
We conclude that 
\[ (x-y)^2 = (z-2)(t-2).\]
On $V(k,l)$,  $S_n(t)S_{m-1}(z)-S_{n-1}(t)S_m(z)=0$. Since $S_m(z)=S_{m-1}(z)$ we get
\[ S_m(z) \Big( S_n(t)-S_{n-1}(t)\Big) =0.\]
Since $z$ is as above, we see that $S_m(z)\neq 0 $ since $S_m^2(z)=\frac1{2-z}$.  Hence $S_n(t)-S_{n-1}(t)=0$. It follows that $t=2\cos\big(\tfrac{(2s+1)\pi}{2n+1}\big)$ (where $0 \le s \le n-1$). We conclude that 
$$(x-y)^2=4\Big(\cos \big(\tfrac{(2r+1)\pi }{2m+1}\big)-1\Big)\Big(\cos\big(\tfrac{(2s+1)\pi}{2n+1}\big)-1\Big).$$ 
This defines $x-y$ explicitly, and therefore determines a set of dimension one in $V(k,l)$.  Since the dimension of $V(k,l)$ is two, this is a codimension one set.

We complete the proof by showing that
on $V(k,l)$, $S_m(z)+S_{m-1}(z)=0$ only for a set of dimension one. 
Note that $z=2\cos\big(\tfrac{2\pi r}{2m+1}\big)$, where $1 \le r \le m$.  We have $S^2_m(z)+S^2_{m-1}(z)-zS_m(z)S_{m-1}(z)=1$. Since $S_m(z)=-S_{m-1}(z)$, we obtain $S^2_m(z)=\frac{1}{2+z}$ and
\begin{eqnarray*}
t &=&  S_m^2(z) \Big(  (x +y )^2 -2z -4\Big)\\
 &=&  \frac{1}{2+z} \Big(  (x +y )^2 -2(z +2)\Big)\\
&=& \frac{(x+y)^2}{2+z}-2.\\
\end{eqnarray*}
We conclude that 
\[ (x+y)^2 = (t+2)(z+2).\]
On $V(k,l)$,  $S_n(t)S_{m-1}(z)-S_{n-1}(t)S_m(z)=0$, so we have 
\[ S_m(z) \Big(S_n(t)+S_{n-1}(t)\Big)=0. \]
Since $z$ is as above, we conclude that $S_n(t)+S_{n-1}(t)=0$.
 This means $t=2\cos\big(\tfrac{2\pi s}{2n+1}\big)$ (where $1 \le s \le n$). Hence
$$(x+y)^2= 4\Big(\cos\big(\tfrac{2\pi r}{2m+1}\big)+1\Big)\Big(\cos\big(\tfrac{2\pi s}{2n+1}\big)+1\Big).$$ 
This defines $x+y$ explicitly, and therefore determines a set of dimension one in $V(k,l)$.  Since the dimension of $V(k,l)$ is two, this is a codimension one set. 
\end{proof}

The following now easily follows.

\begin{proposition}\label{proposition:birational1}
The set $V(k,l) \subset \BC^3(x,y,z)$ is birational to $U(k,l) \subset \BC^3(u,v,z)$.
\end{proposition}

\begin{proof}
As discussed above, the substitution defines a rational map between $V(k,l)$ and $U(k,l)$, namely 
\[ (x,y,z) \mapsto  \Big( \frac{xS_m(z)+yS_{m-1}(z)}{S^2_m(z)-S^2_{m-1}(z)} , \frac{yS_m(z)+xS_{m-1}(z)}{S^2_m(z)-S^2_{m-1}(z)}, z \Big) \]
with   inverse
\[ (u,v,z) \mapsto \Big(uS_m(z)-vS_{m-1}(z), vS_m(z)-uS_{m-1}(z), z\Big).\]
It suffices to see that  $S^2_m(z)-S^2_{m-1}(z)=0$ only for a set of codimension  one on $V(k,l)$, which follows from Lemma~\ref{lemma:zeroset}.  
\end{proof}

We now wish to perform one more birational transformation.
\begin{definition}
Let $W(k,l)$ be the vanishing set of 
\[ S_n(t)S_{m-1}(z)-S_{n-1}(t)S_m(z) \] in $\BC^3(t,v,z)$.

For each odd integer $l$, let $W_0(l,l)$ denote the component of $W(l,l)$ given by $t=z$ and if $|l|>3$ let $W_1(l,l)$ denote the projective closure of the scheme-theoretic complement of $W_0(l,l)$ in $W(l,l).$
\end{definition}

First, we prove a lemma. 
\begin{lemma}\label{lemma:zeroset2}
On $U(k,l)$, $v=0$
 only for a set of dimension zero. 
\end{lemma}  

\begin{proof}  If $v=0$ then since 
$$t=uv-z\big(S^2_m(z)+S^2_{m-1}(z)\big)+4S_m(z)S_{m-1}(z)$$
we conclude that 
$$t=-z\big(S^2_m(z)+S^2_{m-1}(z)\big)+4S_m(z)S_{m-1}(z).$$

The defining polynomial for $U(k,l)$ is $S_n(t)S_{m-1}(z)-S_{n-1}(t)S_m(z)$.
Upon substituting the above polynomial in $\BZ[z]$ for $t$ we see that this defining polynomial can be expressed as a polynomial in $\BZ[z]$.  As a result, this has a finite number of roots.  For each of these $z$ values, there is one associated $t$, and hence we have a finite number of points on $U(k,l)$ where $v=0$. 
\end{proof}

Now we are prepared to show the following.

\begin{proposition}\label{proposition:birational2}
The set $U(k,l) \subset \BC^3(u,v,z)$ is birational to $W(k,l)\subset \BC^3(t,v,z)$.
\end{proposition}

\begin{proof}

Since $t$ is linear in $u$, we define the rational  map from $\BC^3(u,v,z)$ to $\BC^3(t, v, z)$ by this replacement.  That is, define the rational map
\[ (u, v, z) \mapsto  \left( \frac{\big(u+z(S^2_m(z)+S^2_{m-1}(z)\big)-4S_m(z)S_{m-1}(z))}v, v, z\right). \]
which has rational inverse 
\[ (t,v,z) \mapsto \big(tv-z(S^2_m(z)+S^2_{m-1}(z))+4S_m(z)S_{m-1}(z), v, z\big). \]
The result now follows from Lemma~\ref{lemma:zeroset2}.
\end{proof}

\begin{definition}
Let $C(k,l)$ be the curve given by 
$$\{S_n(t)S_{m-1}(z)-S_{n-1}(t)S_m(z)=0\} \subset \BC^2 (t,z).$$
For each odd integer $l$, let $C_0(l,l)$ denote the component of $C(l,l)$ given by $t=z$ and if $|l|>3$ let $C_1(l,l)$ denote the projective closure of the scheme-theoretic complement of $C_0(l,l)$ in $C(l,l).$
\end{definition}

With this definition, the surface $W(k,l)$ is a product of the curve $C(k,l)$ and $\BC$.  We have shown that $V(k,l)$ is birational to $W(k,l)$, which is equivalent to the following, proving the first portion of Theorem~\ref{theorem:summary1}.
\begin{theorem}\label{theorem:VbirationaltoW} 
The algebraic set $X_{irr}(k,l)$ is birational to $W(k,l)$ which is, in turn,  isomorphic to $C(k,l) \times \BC$. 
\end{theorem}

\section{Smoothness and Irreducibility of $W(k,l)$}\label{section:smoothnessandirreducibility}

We will show that if $k\neq l$ then $W(k,l)$ is smooth and irreducible, and if $k=l$ then $W(l,l)$ has two irreducible components. Since $W(k,l)$ is the product of $C(k,l)$ and $\BC$, we will focus on the curve $C(k,l)$.   Our proof is similar to  \cite{MR2827003}, but  with small modifications. Recall that $k=2m+1$ and $l=2n+1$.
The equation $S_n(t)S_{m-1}(z)=S_{n-1}(t)S_m(z)$ can be written as 
\[ \frac{S_n(t)}{S_{n-1}(t)} = \frac{S_m(z)}{S_{m-1}(z)} \]
when $S_{n-1}(t)S_{m-1}(z)\neq 0$.
\begin{definition}
Let $h_j=S_j/S_{j-1}$, $\Delta_j=S'_jS_{j-1}-S_jS'_{j-1}$ and $H_n=S''_jS_{j-1}-S''_{j-1}S_j.$
\end{definition}
We can rewrite the defining equation for $W(k,l)$ as $h_n(t)=h_m(z)$, and  with this notation the derivative is  $h'_j=\Delta_j/ S_{j-1}^2$.

The following lemma can be verified by using Lemma \ref{lemma:chebyshevidentities}. 
\begin{lemma}
We have 

a) $(\omega^2-4)\Delta_j(\omega)=S_{2j}(\omega)-(2j+1).$

b) $(\om^2-4)^2 H_j(\omega)=(2j-2)\omega S_{2j}(\omega)-(4j+2)S_{2j-1}(\omega)+(4j+2)\om$.
\label{H}
\end{lemma}

We will need the following lemma  (see \cite{MR2827003} Lemma 2.6) to connect smoothness and irreducibility.  
\begin{lemma}\label{lemma:smoothimpliesirreducible}
Let $C\subset \BP^1\times \BP^1$ be a smooth projective curve of bidegree $(a,b)$ with $a,b>0$. Then $C$ is irreducible and its genus is $(a-1)(b-1)$. 
\end{lemma}

The proof of smoothness  will follow from comparing valuations at potential critical points.  We begin with a few lemmas. 
In the case that $mn<0$ we use the following lemma.

\begin{lemma}
Let $\omega \in \BC$ be a root of $\Delta_n$. If $n>0$ then $|h_n(\omega)|>1$, and if $n<0$ then $|h_n(\omega)|<1$.
\label{mn<0}
\end{lemma}

\begin{proof}
Suppose that $\Delta_n(\omega)=0$. By Lemma \ref{H}, $S_{2n}(\omega)=2n+1$. We have $S_{n-1}(\omega) \not=0$ (otherwise $S_n(\omega)S'_{n-1}(\omega)=S'_n(\omega)S_{n-1}(\omega)-\Delta_n(\omega)=0$ which cannot occur, since $S_{n-1}$ is separable and relatively prime to $S_n$ in $\BC[\omega]$). Hence $h_n(\omega)=S_n(\omega)/S_{n-1}(\omega)$ is well-defined. Write $\omega = \sigma + \sigma^{-1}$. We have $S_{2n}(\omega)=2n+1$, i.e. $\sigma^{2n+1}-\sigma^{-(2n+1)}=(2n+1)(\sigma-\sigma^{-1})$. Assume $n>0$. Then $\sigma^{2n+1}-\sigma^{-(2n+1)}$ and $\sigma-\overline{\sigma}$ are in the same half-planes. It follows that $\sigma^{2n+1}-\overline{\sigma}^{2n+1}$ and $\sigma-\overline{\sigma}$ are in the same half-plane. Since both these values are purely imaginary, we conclude $(\sigma^{2n+1}-\overline{\sigma}^{2n+1})(\sigma-\overline{\sigma})\le 0$, with equality if and only if $\sigma^{2n+1}$ is real. 

Let $\alpha=\sigma \overline{\sigma}= |\sigma|^2>0$. We have
\begin{eqnarray*}
|\sigma^{n+1}  & - &  \sigma^{-(n+1)}|^2-|\sigma^n-\sigma^{-n}|^2\\
&  = &  (\sigma^{n+1}-\sigma^{-(n+1)})(\overline{\sigma}^{n+1}-\overline{\sigma}^{-(n+1)})-(\sigma^n-\sigma^{-n})(\overline{\sigma}^n-\overline{\sigma}^{-n}) \\
&=& (\alpha^{n+1}+\alpha^{-(n+1)}-(\alpha^{n}+\alpha^{-n}))\\
&& -(\sigma^{2n+1}-\overline{\sigma}^{2n+1})(\sigma-\overline{\sigma})/\sigma^{n+1}\overline{\sigma}^{n+1}\\
&=& (\alpha-1)(\alpha^{2n+1}-1)/\alpha^{n+1}-(\sigma^{2n+1}-\overline{\sigma}^{2n+1})(\sigma-\overline{\sigma})/\alpha^{n+1} \\
&\ge &0.
\end{eqnarray*} 
Equality holds if and only if $|\sigma|^2=\alpha=1$ and $\sigma^{2n+1}$ is real, so if and only if $\sigma^{2n+1}=\pm 1$. If $\sigma^{2n+1}=\pm 1$, then from $\sigma^{2n+1}-\sigma^{-(2n+1)}=(2n+1)(\sigma-\sigma^{-1})$, we find $\sigma = \sigma^{-1}$, so $\sigma = \pm 1$ and $\omega = \pm 2$. If $\omega = \pm 2$ then $|h_n(\omega)|=|S_n(\omega)/S_{n-1}(\omega)|=(n+1)/n>1$. 

The proof for $n<0$ is similar. In that case $\sigma^{2n+1}-\overline{\sigma}^{2n+1}$ and $\sigma-\overline{\sigma}$ are in opposite half-planes and ($\alpha-1)(\alpha^{2n+1}-1) \le 0$.
\end{proof}

In the remaining case ($mn>0$) we can use non-archimedian places instead of complex absolute values. For any root $\omega$ of $\Delta_n$, we have $S_{2n}(\omega)=2n+1$. It follows that 
$$h^2_n(\omega)-1=\left(\frac{S_n(\om)}{S_{n-1}(\om)}\right)^2-1=\frac{S_{2n}(\omega)}{S^2_{n-1}(\om)}=\frac{2n+1}{S^2_{n-1}(\om)}.$$

\begin{lemma}
For any field $\BF$ with characteristic not dividing $2n$, the polynomial $S_{n-1}$ is separable over $\BF$ and we have $(\Delta_n,S_{n-1})=(1)$ in $\BF[\om]$.
\label{main}
\end{lemma}

\begin{proof}
We have $(\sigma^{n+1}-\sigma^{n-1})S_{n-1}=\si^{2n}-1$ and the reduction of this polynomial to $\BF$ is separable. It follows that $S_{n-1}$ is separable over $\BF$, i.e. $(S_{n-1}, S'_{n-1})=(1)$. Since $\Delta_n=S'_nS_{n-1}-S_nS'_{n-1}$, we have $(\Delta_n,S_{n-1})=(S_nS'_{n-1},S_{n-1})=(1).$ 
\end{proof}

\begin{lemma}
Let $p$ be a prime dividing $2n+1$. Let K be a number field containing a root $\om$ of $\Delta_n$. Let $v$ be a valuation on $K$ with $v(p)=1$. Then $v(S_{n-1}(\om))=0$.
\label{eval}
\end{lemma}

\begin{proof}
The polynomial $\Delta_n$ is monic, so $\om$ is an algebraic integer. Let $\fp$ be the prime associated with $v$, and $\BF_{\fp}$ be its residue field. Then the characteristic $p$ of $\BF_{\fp}$ does not divide $2n$, so by Lemma \ref{main} the reduction of $S_{n-1}(\om)$ to $\BF_{\fp}$ is not $0$. This implies $v(S_{n-1}(\om))=0$.
\end{proof}

We now address smoothness.  


\begin{proposition}\label{proposition:Wsmooth1}
Let $k$ and $l$ be any odd integers with $k\neq l$.  Then $C(k,l)$ is smooth over $\BQ$.
\end{proposition}

\begin{proof}

Suppose $P=(t_0,z_0)$ is a singular point on the affine part of $C(k,l)$. Then $S_{n-1}(t_0) \not= 0$ and $S_{m-1}(z_0) \not= 0$. (If $S_{n-1}(t_0)= 0$ then $S_{m-1}(z_0) = 0$. Since $P$ is a singular point, we also have $S'_{n-1}(t_0)= 0$ and $S'_{m-1}(z_0) = 0$. This is impossible since $S_j$ is separable.) Then $C(k,l)$ can be given around $P$ by $h_n(t)=h_m(z)$. The fact that $P$ is a singular point is then equivalent to the fact that $t_0$ and $z_0$ are critical points for $h_n$ and $h_m$ respectively. (We have $\Delta_n(t_0)=\Delta_m(z_0)=0$, i.e. $h'_n(t_0)=h'_m(z_0)=0$.)


First, consider the case when $kl<0$. The points at infinity are smooth by \cite[Lemma 5.6]{MR2827003}. The proposition follows from Lemma~\ref{mn<0}. That is, the values of $h_k$ at its critical points are all different from each other, and they are also different from the values of $h_l$ at all its critical points when $k \not= l$. 

Now, assume that $kl>0$ but $k\neq l$.
Assume $P(t_0,z_0)$ is a singular point over $\overline{\BQ}$ of the standard affine part of $C(k,l)$. Let $K$ be the number field $\BQ(t_0,z_0)$. We have $\Delta_n(t_0)=\Delta_m(z_0)=0$ and $C(k,l)$ is given around $P$ by $h_n(t_0)=h_m(z_0)$. It follows that $h^2_n(t_0)-1=h^2_m(z_0)-1$, i.e.
\begin{equation}\tag{$*$} \frac{2n+1}{S^2_{n-1}(t_0)}=\frac{2m+1}{S^2_{m-1}(z_0)}.\end{equation}
Let $p$ be any prime such that $v_p(2n+1) \not= v_p(2m+1)$. By symmetry we may assume $v_p(2n+1)>v_p(2m+1)$. Let $\fp$ be any prime of $K$ above $p$, and let $v$ be the valuation on $K$ associated to $p$, normalized so that $v$ restricts to $v_p$ on $\BQ$. By Lemma \ref{eval}, we have
$$v\Big(\frac{2n+1}{S^2_{n-1}(t_0)}\Big)=v(2n+1)>v(2m+1) \ge v\Big(\frac{2m+1}{S^2_{m-1}(z_0)}\Big).$$
This contradicts the equality $(*)$, and we conclude that no singular point $P$ exists on the affine part. By \cite[Lemma 5.6]{MR2827003} there are no singular points at infinity. 
\end{proof}


\begin{proposition}\label{proposition:Wsmooth2}
Let $l$ be any odd integer. Then the curve $C_1(l,l)$ is smooth over $\BQ$.
\end{proposition}

\begin{proof}
Let $F=S_n(t)S_{n-1}(z)-S_{n-1}(t)S_n(z)$ and $G=F/(z-t)$. Then $C_1(l,l)$ is defined by $G(t,z)=0$. Any singular point of $C_1(l,l)$ is also a singular point of $C(l,l)$. By \cite[Lemma 5.6]{MR2827003}, we find that $C(l,l)$ is smooth at all points at infinity, so $C_1(l,l)$ is as well. Assume that $P=(t_0,z_0)$ is a singular point of the standard affine part of $C_1(l,l)$. Then $P$ is also a singular point of $C(l,l)$. Note that $\Delta_n(t_0)=0$ and $\Delta_n(z_0)=0$, and we may rewrite $F(P)=0$ as $h_n(t_0)=h_n(z_0)$. Recall $(\om^2-4)\Delta_n(\om)=S_{2n}(\om)-(2n+1)$.

Since $S_n^2(\om)-\om S_n(\om)S_{n-1}(\om)+S_{n-1}^2(\om)=1$ and $S_n^2(\om)-S_{n-1}^2(\om)=S_{2n}(\om)$, we have
$$h_n(\om)+h^{-1}_n(\om)=\om + \frac{1}{S_n(\om)S_{n-1}(\om)}$$
and 
$$h_n(\om)-h^{-1}_n(\om)=\frac{S_{2n}(\om)}{S_n(\om)S_{n-1}(\om)}.$$
Since $h_n(t_0)=h_n(z_0)$ and $S_{2n}(t_0)=S_{2n}(z_0)=2n+1$, we conclude that $t_0=z_0$. 

Recall that $H_n=S''_nS_{n-1}-S_nS''_{n-1}$. By l'Hopital's rule, we have
$$-H_n(t_0)=F_{zz}(t_0,t_0)=\lim_{z \to t_0} \frac{F_z(t_0,z)}{z-t_0}=2\lim_{z \to t_0} \frac{F(t_0,z)}{(z-t_0)^2}=2 \lim_{z \to t_0} \frac{G(t_0,z)}{z-t_0}=2 G_z(t_0,t_0).$$
The fact that $C_1(l,l)$ is singular at $P=(t_0,t_0)$ implies that $0=G_z(P)=-\frac{1}{2}H_n(t_0)$. Hence, by Lemma \ref{H}
$$(2n-2)t_0 S_{2n}(t_0)-(4n+2)S_{2n-1}(t_0)+(4n+2)t_0=(t_0^2-4)^2 H_n(t_0)=0.$$
Since $S_{2n}(t_0)=2n+1$, we obtain $S_{2n-1}(t_0)=nt_0$. Since $S^2_{2n}(t_0)-t_0S_{2n}(t_0)S_{2n-1}(t_0)+S^2_{2n-1}(t_0)=1$, we conclude that $t_0=\pm 2$. This is a contradiction, since $\Delta_n(\pm 2) \not= 0$ by direct calculation. We are done.
\end{proof}

\begin{proposition}\label{proposition:irreducibility}
The algebraic set $C(k,l)$ is smooth and has 1 irreducible component if $k\neq l$. The curve $C(3,3)=C(-3,-3)$ is given by $t=z$.  If $k=l$ and $|l|>3$ then $C(k,l)$  has 2 irreducible components, $C_0(l,l)$ and $C_1(l,l)$. Both of $C_0(l,l)$ and $C_1(l,l)$ are smooth.
\end{proposition}

\begin{proof}
By Lemma~\ref{lemma:smoothimpliesirreducible} it suffices to show that $C(k,l)$ is smooth.  If $k\neq l$, then $C(k,l)$ is smooth by Proposition~\ref{proposition:Wsmooth1}.  If $k=l$ then $C_1(l,l)$ is smooth by Proposition~\ref{proposition:Wsmooth2}. The proposition follows since $C_0(l,l)$ is given by $t=z$ and  is smooth. 
\end{proof}

\begin{figure}[h!]
\begin{center}
\includegraphics[scale=.5]{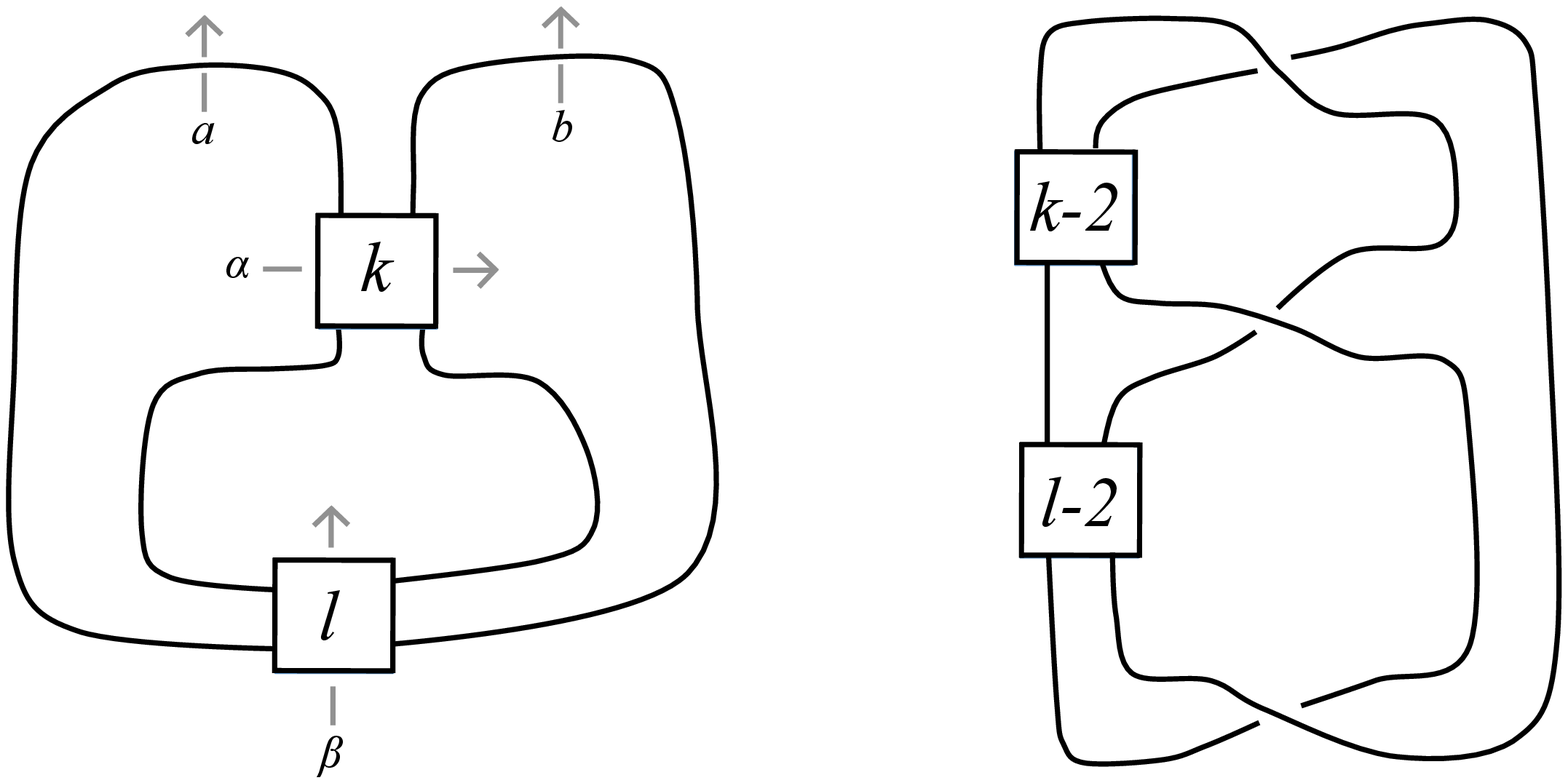}
\end{center}
\vspace{-4cm}
\caption{Meridian loops on double twist links and the four plat presentation.}
\label{fig:4plat}
\end{figure}

We have shown that if $k\neq l$ then $X_{irr}(k,l)$ is a single irreducible component.  When $k=l$ and $|l|>3$, we have shown that 
$X_{irr}(k,l)$ is comprised of two irreducible components, and we now identify the canonical component. 

\begin{lemma}\label{lemma:canonicalcomponent} When $k\neq l$  then $X_0(k,l)$ is birational to $C_0(k,l) \times \BC$. 
The curve $C(3,3)=C(-3,-3)$ is given by $t=z$ and $X_0(3,3)$ is birational to $C(3,3) \times \BC$.
When $k=l$ and $|l|>3$ then $X_0(l,l)$ is birational to $C_0(l,l)\times \BC$ and there is one more irreducible component of $X_{irr}(l,l)$, birational to $C_1(l,l)\times \BC$.
\end{lemma}

\begin{proof} By Theorem~\ref{theorem:VbirationaltoW} $X_{irr}(k,l)$ is birational to $C(k,l)\times \BC$. 
When $k\neq l$  Proposition~\ref{proposition:irreducibility} shows that $X_0(k,l)=X_{irr}(k,l)$.    

By the definition of $C_0(l,l)$ it suffices to show that $t=z$ corresponds to the canonical component.
 By construction, $z=\chi_{\rho}(ab^{-1})$  corresponds to the loop  $\alpha$ pictured in Figure~\ref{fig:4plat}.  Moreover, $t=\chi_{\rho}(w_k)$ corresponds to the loop $\beta$ pictured in the figure.  When $k=l$ the symmetry induced by flipping the four plat upside down swaps these loops and so on the level of the character variety induces $t=z$. This must hold for any discrete faithful representation, and all Dehn fillings.  By work of Thurston \cite{Thurston}, all but finitely many of these are on canonical components, and so are dense. (See \cite{MR1248091} and also \cite{MR2827003} Section 2.3.)  The fact that there are exactly two irreducible components in this case follows from   Proposition~\ref{proposition:irreducibility}.
\end{proof}

We summarize this section in the following theorem. 
\begin{reptheorem}{theorem:summary1}{\em
 Let $k=2m+1$ and $l=2n+1$. The algebraic set $X_{irr}(k,l)$ is  birational to $C(k,l)\times \BC$ where the curve $C(k,l)\subset \BC^2(t,z)$ is given by 
\[ C(k,l) = \{S_n(t)S_{m-1}(z)-S_{n-1}(t)S_m(z)=0\}.\]

If $k\neq l$ then $C(k,l)$ is smooth and irreducible as considered in $\BP^1(t)\times \BP^1(z)$ and $X_0(k,l)=X_{irr}(k,l)$ is birational to $C(k,l)\times \BC$.  

The curve $C(3,3)=C(-3,-3)$ is given by $t=z$. If $k=l$ and $|l|>3$ then $C(l,l)$ is the union of exactly two components, $C_0(l,l)$ given by $t=z$ and $C_1(l,l)$, the scheme-theoretic complement of $C_0(l,l)$ in $C(l,l)$. Both are smooth and irreducible  as considered in $\BP^1(t)\times \BP^1(z)$. The algebraic set $X_{irr}(l,l)= X_0(l,l) \cup X_1(l,l)$ where  $X_0(l,l)$ is birational to  $ C_0(l,l)\times \BC$ and $X_1(l,l)$ is birational to $ C_1(l,l)\times \BC$. }
\end{reptheorem}

We conclude this section with a few remarks about symmetries. 
The proof of Lemma~\ref{lemma:canonicalcomponent} relied on analysis of the symmetry which flips the four plat upside down.  For all $k$ and $l$, the link complement $S^3-J(k,l)$ has non-trivial symmetry group.  In the case when $k\neq l$ this is generated by the flip about a vertical axis through the $k$ half twists, and the analogous symmetry through an axis through the $l$ half twists.  (In Figure~\ref{fig:4plat} this axis is a circle through the middle of the $l$ half twists and going horizontally through the $k$ box.)  These symmetries both take the loop corresponding to $a$ to a loop freely homotopic to the  loop corresponding to $b^{-1}$, and the free homotopy class of the un-oriented  loop corresponding to $ab^{-1}$ is fixed.  The effect on the character variety is that $x=\chi_{\rho}(a)$ is sent to $y=\chi_{\rho}(b)=\chi_{\rho}(b^{-1})$  and $z=\chi_{\rho}(ab^{-1})$ is fixed.  The variable $t=\chi_{\rho}(w_k) $ is symmetric in $x$ and $y$, and so is $\varphi$. The effect of these symmetries on the defining equation, $\varphi$, of the character variety is trivial.  These symmetries are only reflected in the symmetry of the defining equations themselves. We conclude that even the non-geometric representations algebraically preserve this symmetry. However, when $k=l$ the additional symmetry fixes the un-oriented free homotopy class of loops corresponding to $a$ and also for $b$, but takes the un-oriented loop corresponding to $ab^{-1}$ to a loop freely homotopic to one corresponding to $\om_k$.  This is not freely homotopic to $(ab^{-1})^{\pm1}$.   It is this that induces the factoring of the defining equation, $\varphi$. In this case, when $|l|>3$ there is a component which corresponds to necessarily non-geometric representations which do not algebraically preserve this symmetry.

\section{Further Invariants} 

We have established  in Theorem~\ref{theorem:summary1} that when $k\neq l$, $X_0(k,l)$ is birational to $C_0(k,l)\times \BC$, and that $C_0(k,l)$ is smooth and irreducible in $\BP^1\times \BP^1$.  We have also shown that $X_{irr}(l,l)$ is birational to the union of $C_0(l,l)\times \BC$ and $C_1(l,l) \times \BC$. 
We now compute the genus of these curves, and the degree of irrationality of $X_0(k,l)$ and $X_1(l,l)$.

\begin{lemma}\label{lemma:bidegree}
When $k\neq l$  the bidegree of $C(k,l)$ is $(\lfloor  \tfrac{|k|}2 \rfloor, \lfloor  \tfrac{|l|}2 \rfloor)$.  The bidegree of $C_1(l,l)$ is  $(\lfloor  \tfrac{|l|}2 \rfloor -1, \lfloor  \tfrac{|l|}2 \rfloor -1)$.
\end{lemma}

\begin{proof}

By Lemma~\ref{lemma:chebyshevidentities},    $S_{-1}=0$ and the degree of $S_j$ is $j$ when $j>0$ and $-j-2$ when $j<-1$. 
Therefore, the bidegree of $C(k,l)$ is $(a,b)$ where $a=n$ if $n>0$ and $a=-n-1$ if $n<-1$ and $b=m$ if $m>0$ and $b=-m-1$ if $m<-1$. This is equivalent to $a=\lfloor  \tfrac{|k|}2 \rfloor $ and $b=\lfloor  \tfrac{|l|}2 \rfloor.$  The computation for $C_1(l,l)$ follows from this using the definition of $C_1(l,l)$. 
\end{proof}

\begin{reptheorem}{theorem:genus}{\em  Let $|k|, |l|>1$. When $k\neq l$  the genus of $C(k,l)$ is $(\lfloor  \tfrac{|k|}2 \rfloor -1)( \lfloor  \tfrac{|l|}2 \rfloor-1)$.  The genus of $C_0(l,l)$ is zero, and for $|l|>3$ the genus of $C_1(l,l)$ is  $(\lfloor  \tfrac{|l|}2 \rfloor -2)^2$.}
\end{reptheorem}

\begin{proof}The result follows from the following, by Lemma~\ref{lemma:bidegree}. If $C$ is a smooth projective curve in $\BP^1\times \BP^1$ of bidegree $(a,b)$ then the genus is $(a-1,b-1)$ (see \cite{MR2827003}). 
\end{proof}

\begin{definition}
Let $X$ be an irreducible (affine or projective) complex variety of dimension $n$.  The degree of irrationality of $X$, $\gamma(X)$ its the minimal degree of any rational map from $X$ to a dense subset of $\BC^n$.  When $X$ is a curve, this is also called the gonality of $X$. 
\end{definition}

The gonality, in its relation to character varieties and Dehn filling is discussed at length in \cite{PetersenReid}.  Moreover, the gonality of the components of the $\SL_2(\BC)$ and $\PSL_2(\BC)$ character varieties are computed (Theorem 9.2, Theorem 9.4). We now compute the degree of irrationality of our sets.

\begin{reptheorem}{theorem:degreeofirrationality}{\em
The degree of irrationality of $X_0(k,l)$ is $\min\{  \lfloor  \tfrac{|k|}2 \rfloor, \lfloor  \tfrac{|l|}2 \rfloor \}$ when $k\neq l$.  The degree of irrationality of $X_0(l,l)$ is 1, and the degree of irrationality of $X_1(l,l)$ is $\lfloor  \tfrac{|l|}2 \rfloor -1$.}
\end{reptheorem}

\begin{proof}

The degree of irrationality of a surface of the form $C\times \BC$ is equal the gonality of $C$
\cite{MR1327053} (Prop 1) and \cite{MR1337188}. (If $C$ is a non-singular projective curve then $C\times \BC$ is a non-singular projective surface since the fibers have genus zero.) Following \cite{PetersenReid} (Lemma 9.1) if $C$ is a smooth irreducible curve in $\BP^1\times \BP^1$ of bidegree $(a,b)$ with $ab\neq 0$ then the gonality of $C$ is $\min\{a,b\}$. The result now follows from Lemma~\ref{lemma:bidegree}.
\end{proof}

\section{Desingularization of  $X_0(3,2m+1)$} 

The simplest subfamily of the hyperbolic 2-component double twist links is when $k=3$ (so $n=1$). This family includes the Whitehead link $5_1^2=(8/3)$ which is  $J(3,3)$, and $6_2^2=(10/3)$ which is $J(3,-3)$.  
In this section we first determine the singular points of the natural model of $X_0(3,l)$ where $l=2m+1$ in Proposition~\ref{proposition:singularpoints}.  In Proposition~\ref{proposition:degeneratefibers}  we determine the degenerate fibers of the map $\phi: S \to \BP^1$, $(x:y:u,z:w) \mapsto (z:w)$.  We then show in Theorem~\ref{theorem:blowup} that the desingularization of the natural model for $X_0(3,2m+1)$ is a series of blowups of $\BP^1\times \BP^1$.  

By Theorem~\ref{theorem:summary1}, $X_0(3,l)=X_{irr}(3,l)$ is birational to $C(3,l)\times \BC$ where $C(3,l)$ is given by $tS_{m-1}(z)=S_m(z)$ in $\BC^2(t,z)$. Since this defining polynomial is linear in $t$ we conclude that $C(3,l)$ is itself birational to $\BC$ and $X_0(3,l)$ is indeed birational to $\BC^2$.  The Whitehead link, $J(-3,-3)=J(3,3)$ is a degenerate case of the $J(3,l)$ links, where $X_0(3,3)=X_{irr}(3,3)$ is given by $t=z$ up to birational equivalence.  

We begin by homogenizing the defining polynomial for $X_0(3,l)$, where $l=2m+1$. 
Recall that
$$t=\left( xS_m(z)-yS_{m-1}(z) \right)\left( yS_m(z)-xS_{m-1}(z) \right)-z\left(S^2_m(z)+S^2_{m-1}(z)\right)+4S_m(z)S_{m-1}(z).$$
Since $S^2_m(z)+S^2_{m-1}(z)-z S_m(z)S_{m-1}(z)=1$, this simplifies to 
\[ t=  xy-z+(xyz+4-x^2-y^2-z^2)S_m(z)S_{m-1}(z). \]
The defining polynomial for the natural model of $X_0(3,l)$ is $tS_{m-1}(z)-S_m(z)$ in $\BC[x,y,z]$. We now homogenize it. 
\begin{definition}
Let $T_j=T_j(z,w)=w^jS_j(\frac{z}{w})$. 
\end{definition}
The following is a direct consequence of the Chebyshev identity $S^2_j(\omega)+S^2_{j-1}(\omega)-\omega S_j(\omega)S_{j-1}(\omega)=1$. 
\begin{lemma}
\label{T}
We have 
$$T_j^2+w^2T_{j-1}^2-z \, T_j T_{j-1}=w^{2j}.$$
\end{lemma}

It is now elementary to determine the homogenous defining polynomial. 
\begin{lemma}\label{lemma:homoganeouspolynomial}
The homogenization of the defining polynomial 
$tS_{m-1}(z)-S_m(z)$ 
in $\BP^2 \times \BP^1=\{ (x:y:u,z:w) \} $ is
\[
F = \left[ (xyw-u^2z)w^{2m}+(xyzw+4u^2w^2-x^2w^2-y^2w^2-u^2z^2)T_mT_{m-1} \right] T_{m-1}-u^2w^{2m}T_m.
\]
\end{lemma}

 We now determine the singular points in the projective closure of our natural model in $\BP^2\times \BP^1$.  To find singular points, we consider solutions $(x:y:u,z:w)$ of $F=F_x=F_y=F_u=F_z=F_w=0$.  

First, we compute these partial derivatives, which is elementary to verify by direct calculations. 
\begin{lemma} With $F$ as in Lemma~\ref{lemma:homoganeouspolynomial} the first order partials of $F$ are given by the following.
\begin{eqnarray*}
F_x &=& \left(yw^{2m}+(yz-2xw)T_mT_{m-1} \right) w \,T_{m-1},\\
F_y &=& \left(xw^{2m}+(xz-2yw)T_mT_{m-1} \right) w \,T_{m-1},\\
F_u &=& -2u \left[   \left( zw^{2m}+(z^2-4w^2)T_mT_{m-1} \right) T_{m-1} + w^{2m}T_m \right],\\
F_z &=& \left[ -u^2w^{2m}+(xyw-2u^2z)T_mT_{m-1} +(xyzw+4u^2w^2-x^2w^2-y^2w^2-u^2z^2)(T_mT_{m-1} )_z \right] T_{m-1}\\
       && + \left[ (xyw-u^2z)w^{2m}+(xyzw+4u^2w^2-x^2w^2-y^2w^2-u^2z^2)T_mT_{m-1} \right] (T_{m-1})_z - u^2w^{2m}(T_m)_z,\\
F_w &=&  \big[ (2m+1)xyw^{2m}-2mu^2zw^{2m-1}+(xyz+8u^2w-2x^2w-2y^2w)T_mT_{m-1} \\
       && + \, (xyzw+4u^2w^2-x^2w^2-y^2w^2-u^2z^2)(T_mT_{m-1} )_w \big] T_{m-1}\\
       && + \left[ (xyw-u^2z)w^{2m}+(xyzw+4u^2w^2-x^2w^2-y^2w^2-u^2z^2)T_mT_{m-1} \right] (T_{m-1})_w\\
       && - u^2(2mw^{2m-1}T_m+w^{2m}(T_m)_w).
\end{eqnarray*}
\end{lemma}

We can now determine the singular points. 
\begin{proposition}\label{proposition:singularpoints} The singular points $(x:y:u,z:w) \in \BP^2 \times \BP^1$ of $F$ are

\begin{itemize}
\item $(1:0:0,1:0),\, (0:1:0,1:0)$,

\item $(1:0:0,z:1),\, (0:1:0,z:1)$ where $z$ is a root of $S_{m-1}(z)$,

\item $(1:1:0,z:1)$ where $z$ is a root of $S_m(z)-S_{m-1}(z)$,

\item $(1:-1:0,z:1)$ where $z$ is a root of $S_m(z)+S_{m-1}(z)$.
\end{itemize}
The number of singularities is $4m$ if $m \ge 1$, and is $-(2+4m)$ if $m \le -2$.
\end{proposition}

\begin{proof}
We break the analysis down into  cases. 

First, we consider the case when $(w:z)=(0:1)$. We have $F_x=F_y=0$, $F=-u^2$ and $F_u=-2u$. Hence $u=0$. Now we have $F_z=0$ and $F_w=xy$. Thus $xy=0$. In this case, there are 2 singular points $(1:0:0,1:0)$ and $(0:1:0,1:0)$.

Next, we consider the case when $w=1$.  First we assume that $S_{m-1}(z)=0$. Then $F_x=F_y=0$, $F=-u^2S_m(z)$, $F_u=-2uS_m(z)$. Since $S_m(z) \not= 0$, we have $u=0$. Then 
 $F_z=xyS'_{m-1}(z)$ and $F_w=xy(T_{m-1})_w$. Since $S'_{m-1}(z) \not= 0$, we must have $xy=0$. In this case, singular points are $(1:0:0,z:1),\, (0:1:0,z:1)$ where $z$ is a root of $S_{m-1}(z)$.

Finally, we assume that $w=1$ and $S_{m-1}(z) \not= 0$. We have 
\begin{eqnarray*}
F_x &=& y+(yz-2x)S_m(z)S_{m-1}(z)=y(S^2_m(z)+S^2_{m-1}(z))-2xS_m(z)S_{m-1}(z),\\
F_y &=& x+(xz-2y)S_m(z)S_{m-1}(z)=x(S^2_m(z)+S^2_{m-1}(z))-2yS_m(z)S_{m-1}(z).
\end{eqnarray*}
Note that if $x$ and $y$ are not simultaneously equal to 0, we must have $S^2_m(z)-S^2_{m-1}(z)=0$. 

We first consider the subcase when $ x=y=0$, so  $(x:y:u)=(0:0:1)$. Then, by Lemma \ref{identities} $$F=S_{m-1}(z) \Big( -z+(4-z^2)S_{m-1}(z)S_{m}(z) \Big)-S_m(z)=-S_{3m}(z).$$ Since $S_{3m}(z)$ is separable in $\BC[z]$, there are no singular points in this case.

Therefore, we may assume that $xy\neq 0$ and $S^2_m(z)-S^2_{m-1}(z)=0$. We consider the cases that $S_m(z)-S_{m-1}(z)=0$ and $S_m(z)+S_{m-1}(z)=0$ separately.

First assume that  $S_m(z)-S_{m-1}(z)=0$. Then $F_x=F_y=0$ is equivalent to $x=y$. Since $S^2_{m}(z)=\frac{1}{2-z}$, we have $F=u^2S_m(z)$  and $F_u=2uS_m(z)$. Hence $u=0$. Now we have 
\begin{eqnarray*}
F_z &=& \left[ S_m(z)S_{m-1}(z) +(z-2)(S_m(z)S_{m-1}(z))' \right] x^2 S_{m-1}(z).
\end{eqnarray*}
From $S^2_m(z)+S^2_{m-1}(z)-zS_m(z)S_{m-1}(z)=1$ and $S_m(z)=S_{m-1}(z)$, we get $(z-2)(S'_m(z)+S'_{m-1}(z))=-S_m(z)$. It follows that $F_z=0$. 
We have
$$F_w = \big[ (2m+1)+(z-4)S_m(z)S_{m-1}(z) +(z-2)(T_mT_{m-1} )_w \big] x^2 S_{m-1}(z).$$
From $T_m^2+w^2T_{m-1}^2-z \, T_mT_{m-1}=w^{2m}$ (by Lemma \ref{T}) and $S_m(z)=S_{m-1}(z)$, we get 
\[(2-z)\Big((T_m)_w+(T_{m-1})_w\Big)S_m(z)+2S^2_m(z)=2m.\]
It follows that $$(2m+1)+(z-4)S_m(z)S_{m-1}(z) +(z-2)(T_mT_{m-1} )_w=1+(z-2)S^2_m(z)=0.$$ Hence $F_w=0$. The corresponding singular points are $(1:1:0,z:1)$ where $z$ is a root of $S_m(z)-S_{m-1}(z)$.

Finally, assume that $xy=0$ and  $S_m(z)+S_{m-1}(z)=0$. Similar to the above, singular points are $(1:1:0,z:1)$ where $z$ is a root of $S_m(z)+S_{m-1}(z)$.
\end{proof}

\begin{definition}
Let $S=\CZ(F) \subset\BP^2 \times \BP^1$ be the vanishing set of $F$ and $\tilde{S}$ be the desingularization of $S$. 
\end{definition}
Now we  determine the degenerate fibers; we determine all $(z:w) \in \BP^1$ such that $F=F_x=F_y=F_u=0$ has at least one solution $(x:y:u) \in \BP^2$.

\begin{proposition}\label{proposition:degeneratefibers} The degenerate fibers of $\phi: S \to \BP^1$, $(x:y:u,z:w) \mapsto (z:w)$, are

\begin{itemize}
\item $\phi^{-1}(1:0) = \{(x:y:u) \in \BP^2 \mid u^2=0\}$,

\item $\phi^{-1}(z:1) = \{(x:y:u) \in \BP^2 \mid u^2=0\}$ where $z$ is a root of $S_{m-1}(z)$,

\item $\phi^{-1}(z:1) = \{(x:y:u) \in \BP^2 \mid \left( xS_m(z)-yS_{m-1}(z) \right) \left( yS_m(z)-xS_{m-1}(z) \right) =0\}$ where $z$ is a root of $S_{3m}(z)$,

\item $\phi^{-1}(z:1) = \{(x:y:u) \in \BP^2 \mid (x-y)^2-(2-z)u^2=0\}$ where $z$ is a root of $S_m(z)-S_{m-1}(z)$,

\item  $\phi^{-1}(z:1) = \{(x:y:u) \in \BP^2 \mid (x+y)^2-(2+z)u^2=0\}$ where $z$ is a root of $S_m(z)+S_{m-1}(z)$.
\end{itemize}
\end{proposition}

\begin{proof} We break the analysis down into  cases. 

First, we consider the case when $(w,z)=(0:1)$.
 We have $F_x=F_y=0$, $F=-u^2$ and $F_u=-2u$. Hence $u=0$. Note that $\phi^{-1}(1:0) = \{(x:y:u) \in \BP^2 \mid u^2=0\}$.

Next, we consider the case when $w=1$. First we assume that $S_{m-1}(z)=0$. Then $F_x=F_y=0$, $F=-u^2S_m(z)$, $F_u=-2uS_m(z)$. Hence $u=0$. In this case $\phi^{-1}(z:1) = \{(x:y:u) \in \BP^2 \mid u^2=0\}$.

Finally, we assume that $w=1$ and $S_{m-1}(z) \not= 0$.  Note that if $x$ and $y$ are not simultaneously equal to 0, we must have $S^2_m(z)-S^2_{m-1}(z)=0$.

We first consider the subcase when $ x=y=0$, so  $(x:y:u)=(0:0:1)$.  Then $F_x=F_y=0$, $F=-S_{3m}(z)$, $F_u=-2S_{3m}(z)$. Hence $S_{3m}(z)=0$. In this case $$\phi^{-1}(z:1) = \{(x:y:u) \in \BP^2 \mid \left( xS_m(z)-yS_{m-1}(z) \right) \left( yS_m(z)-xS_{m-1}(z) \right) =0\}.$$
As a result we may assume that $xy\neq 0$.  Therefore $S_m(z)-S_{m-1}(z)=0$ or $S_m(z)+S_{m-1}(z)=0$. If  $S_m(z)-S_{m-1}(z)=0$ then $F=F_x=F_y=F_u=0$ is equivalent to $x=y$ and $u=0$. In this case $\phi^{-1}(z:1) = \{(x:y:u) \in \BP^2 \mid (x-y)^2-(2-z)u^2=0\}$.  If $S_m(z)+S_{m-1}(z)=0$ then $F=F_x=F_y=F_u=0$ is equivalent to $x=-y$ and $u=0$. In this case $\phi^{-1}(z:1) = \{(x:y:u) \in \BP^2 \mid (x+y)^2-(2+z)u^2=0\}$.
\end{proof}

Next, we consider desingularization. Since $S$ is birational to $\BP^1 \times \BP^1$, we can blow down $\tilde{S}$ over $\BP^1$ some number of times so that it becomes a fiber bundle $\BP^1 \times \BP^1$ over $\BP^1$.

\begin{definition}
In the following, let $\chi$ denote the Euler characteristic of a surface. 
Let $S_{\text{sing}}$ be the set of singular points of $S$ and $N_{\text{sing}}=|S_{\text{sing}}|$. 
Furthermore, let $N$ be such that  $\tilde{S}$ is obtained from $\BP^1 \times \BP^1$ by $N$ one-point blow ups.  
\end{definition}

We have $$\chi(\tilde{S})=\chi(S - S_{\text{sing}})+N_{\text{sing}} \, \chi(\BP^1)=\chi(S)+N_{\text{sing}}$$ 
(see \cite{harada} Lemma 2.2).

By definition $\tilde{S}$ is obtained from $\BP^1 \times \BP^1$ by $N$ one-point blow ups. Then  since $\chi(\BP^1 \times \BP^1)=4$, using $ \BP^1 \times \BP^1$ in place of $S$ in the above, we have 
 $$\chi(\tilde{S})=\chi(\BP^1 \times \BP^1)+N=4+N.$$ 
It follows that $N=\chi(S)+N_{\text{sing}}-4$.
We summarize this as a lemma. 
\begin{lemma}\label{lemma:singularitiesandblowups} We have
 $N=\chi(S)+N_{\text{sing}}-4$.
\end{lemma}

\begin{proposition}\label{proposition:eulerchar}The Euler characteristic of $S$ is 
$ \chi(S)=\begin{cases} 4+5m &\mbox{if } m \ge 1 \\ 
                          -5m & \mbox{if } m \le -2. \end{cases}$
\end{proposition}

\begin{proof}

 Let $\varphi: S \hookrightarrow \BP^2 \times \BP^1 \dashrightarrow \BP^1 \times \BP^1$ be the rational map defined by $(x:y:u, z:w) \mapsto (x:y,z:w)$. Let $P$ be the set of points $(0:0:1,z:1)$ where $z$ is a root of $S_{3m}(z)$. 
The map $\varphi$ is not defined at points in $P$. Let $U:=S - P$. We now determine $\varphi(U)$. 

Write $F=G+u^2H$ where 
\begin{eqnarray*}
G &=& \left( xyw^{2m+1}+(xyzw-x^2w^2-y^2w^2)T_mT_{m-1} \right) T_{m-1},\\
H &=& \left( -zw^{2m}+(4w^2-z^2)T_mT_{m-1} \right) T_{m-1}-w^{2m}T_m.
\end{eqnarray*}
Note that $\varphi(U)$ is the collection of all points $(x:y,z:w) \in \BP^1 \times \BP^1$ except those for which $F(x:y,z:w) \in \BC[u]$ is a nonzero constant. The polynomial $F(x:y,z:w) \in \BC[u]$ is a nonzero constant whenever $H=0$ and $G \not= 0$, which is equivalent to $$w=1,~S_{3m}(z)=0, \text{~and~}(xS_m(z)-yS_{m-1}(z))(yS_m(z)-xS_{m-1}(z)) \not= 0.$$
Hence $\varphi(U)=\BP^1 \times \BP^1 - Q$, where $Q$ is the set of points $(x:y,z:1) \in \BP^1 \times \BP^1$ satisfying $S_{3m}(z)=0$ and $(xS_m(z)-yS_{m-1}(z))(yS_m(z)-xS_{m-1}(z)) \not= 0.$ Note that $\chi(Q)=0$.

Let $L$ be the set of points $(x:y,z:1) \in \BP^1 \times \BP^1$ satisfying $S_{3m}(z)=0$ and $(xS_m(z)-yS_{m-1}(z))(yS_m(z)-xS_{m-1}(z)) = 0.$ Note that $\{G=H=0\} \subset \BP^1 \times \BP^1$ is equal to $L$. Hence $$\chi(L)=\chi(\varphi^{-1}(L))=\begin{cases} 6m &\mbox{if } m \ge 1 \\ 
-(6m+4) & \mbox{if } m \le -2. \end{cases}$$

Recall that $G=\left( xyw^{2m}+(xyz-x^2w-y^2w)T_mT_{m-1} \right) w \, T_{m-1}$. Since $T_m^2+w^2T_{m-1}^2-z \, T_mT_{m-1}=w^{2m}$, we have $G=(x \, T_m-yw \, T_{m-1})(y \, T_m-xw \, T_{m-1}) w \, T_{m-1}$.

Let $B:=\CZ(G)$ be the zero set of $G$ in $\BP^1 \times \BP^1$. Then $B=B_1 \cup B_2 \cup B_3$ where
\begin{eqnarray*}
B_1 &=& \CZ(w) = \BP^1 \times \{(1:0)\},\\
B_2 &=& \CZ(T_{m-1}) = \BP^1 \times \{(z:1) \mid S_{m-1}(z)=0\},\\
B_3 &=& \CZ(x \, T_m-yw \, T_{m-1}) \cup \CZ(y \, T_m-xw \, T_{m-1})
\end{eqnarray*}
are subsets in $\BP^1 \times \BP^1$. 

We have $B_3=B_{31} \cup B_{32}$, where $B_{31}=\CZ(x \, T_m-yw \, T_{m-1})$ and $B_{32}=\CZ(y \, T_m-xw \, T_{m-1})$. Note that $(x:y,z:w) \in B_{31} \cap B_{32}$ if and only if $x=y \text{~and~} T_m=w \, T_{m-1}$, or $x=-y \text{~and~} T_m=-w \, T_{m-1}$. Hence $$B_{31} \cap B_{32}=\{(1:1,z:1) \mid S_m(z)-S_{m-1}(z)=0\} \cup \{(1:-1,z:1) \mid S_m(z)+S_{m-1}(z)=0\}.$$ 
It follows that $\chi(B_{31} \cap B_{32})=\begin{cases} 2m &\mbox{if } m \ge 1 \\ 
-(2m+2) & \mbox{if } m \le -2 \end{cases}.$ Then 
$$\chi(B_3)=\chi(B_{31})+\chi(B_{32})-\chi(B_{31} \cap B_{32})
=\begin{cases} 4-2m &\mbox{if } m \ge 1 \\ 
6+2m & \mbox{if } m \le -2 \end{cases}.$$.

We have $B_1 \cap B_2 =\emptyset$, $B_1 \cap B_3=\{(1:0,1:0), (0:1, 1:0)\}$, and 
$$B_2 \cap B_3=\{(1:0,z:1),(0:1,z:1) \mid S_{m-1}(z)=0\}.$$
Hence
\begin{eqnarray*}
\chi(B) &=& \chi(B_1)+\chi(B_2)+\chi(B_3)-\chi(B_1 \cap B_2)-\chi(B_1 \cap B_3)-\chi(B_2 \cap B_3)+\chi(B_1 \cap B_2 \cap B_3)\\
&=& \begin{cases} 2+(2m-2)+(4-2m)-0-2-(2m-2)+0=4-2m &\mbox{if } m \ge 1 \\ 
2-(2m+2)+(6+2m)-0-2+(2m+2)+0=6+2m & \mbox{if } m \le -2 \end{cases}.
\end{eqnarray*}
It follows that
\begin{eqnarray*}
\chi(U) &=& 2\chi(\BP^1 \times \BP^1 - (B \sqcup Q))+\chi(B-L)+\chi(\varphi^{-1}(L)) \\
        &=& 2\chi(\BP^1 \times \BP^1)-\chi(B)-2\chi(Q)-\chi(L)+\chi(\varphi^{-1}(L)) \\
        &=& \begin{cases} 4+2m &\mbox{if } m \ge 1 \\ 
                          2-2m & \mbox{if } m \le -2. \end{cases}
\end{eqnarray*}
Then \[\chi(S)=\chi(U)+\chi(P)=\begin{cases} (4+2m)+3m=4+5m &\mbox{if } m \ge 1 \\ 
                          (2-2m)-(3m+2)=-5m & \mbox{if } m \le -2. \end{cases}.\]                          
\end{proof}

Proposition~\ref{proposition:eulerchar} and  Proposition~\ref{proposition:singularpoints} along with the fact that 
\[ N=\chi(S)+N_{\text{sing}}-4\]
gives 
 $N=\chi(S)+N_{\text{sing}}-4=\begin{cases} (4+5m)+4m-4=9m &\mbox{if } m \ge 1 \\ 
                          (-5m)+(-(2+4m))-4=-(6+9m) & \mbox{if } m \le -2. \end{cases}$ 

\begin{reptheorem}{theorem:blowup}{\em
The desingularization of the canonical component of the $\SL_2(\BC)$-character variety of the double twist link $J(3, 2m+1)$ is the conic bundle over the projective line $\BP^1$ which is isomorphic to the surface obtained from $\BP^1 \times \BP^1$ by repeating a one-point blow up $9m$ times if $m \ge 1$, and $-(6+9m)$ times if $m \le -2$. Equivalently,  it is isomorphic to the surface obtained from  $\BP^2$ by repeating a one-point blow up $1+9m$ times if $m \ge 1$, and $-(5+9m)$ times if $m \le -2$.
\label{n=1}}
\end{reptheorem}



\bibliographystyle{amsplain}
\bibliography{doubletwistlinks}

\end{document}